\documentclass{article}

\usepackage[backend=bibtex, 
sortcites=true,
firstinits=true,
hyperref,
maxbibnames=99,
]{biblatex}

\bibliography{Bibliography}{}

\usepackage{enumerate}

\usepackage[centertags]{amsmath}
\usepackage{amsfonts}
\usepackage{amssymb}
\usepackage{amsthm}
\usepackage{newlfont}
\usepackage{mathtools}
\usepackage[top=1in, bottom=1in, left=1in, right=1in]{geometry}
\usepackage{tikz}

\theoremstyle{plain}
\newtheorem{theorem}{Theorem}[section]

\newtheorem{corollary}[theorem]{Corollary}
\newtheorem{lemma}[theorem]{Lemma}

\theoremstyle{definition}

\newtheorem{example}{Example}[section]
\newtheorem*{remark}{Remark}

\newcommand{\R}{\mathbb{R}}

\newcommand{\eps}{\varepsilon}

\newcommand{\Hidden}[1]{}

\DeclareMathOperator{\vol}{vol}

\begin{document}

\title{Li-Yau inequality under $CD(0,n)$ on graphs}
\author{
Florentin M\"unch\footnote{MPI MiS Leipzig, muench@mis.mpg.de}
}
\date{\today}
\maketitle

\begin{abstract}
We introduce a modified non-linear heat equation $\partial_t u = \Delta u + \Gamma u$ as a substitute of $\log P_t f$ where $P_t$ is the heat semigroup.
We prove an exponential decay of $\Gamma u$ under the Bakry Emery curvature condition $CD(K,\infty)$ and prove the Li-Yau inequality $-\Delta u_t \leq \frac{n}{2t}$ under the Bakry Emery curvature condition $CD(0,n)$.
From this, we deduce the volume doubling property which solves a major open problem in discrete Ricci curvature. As an application, we show that there exist no expander graphs satisfying $CD(0,n)$.
\end{abstract}

\tableofcontents

\section{Introduction}

The celebrated Li-Yau inequality  on $n$-dimensional Riemannian manifolds with non-negative Ricci curvature states
\[
-\Delta \log P_t f \leq \frac{n}{2t}
\]
for all positive $f$ where $P_t = e^{\Delta t}$ denotes the heat semigroup \cite{li1986parabolic}.

Tremendous effort has been invested to establish a Li-Yau inequality on graphs with non-negative Ricci curvature \cite{bauer2015li,horn2014volume,
munch2014li,dier2017discrete,gong2018li}. 
To this end, there have been introduced several non-linear modifications of the Bakry-Emery curvature. These modified curvature notions are hard to compute or estimate in practice. Additionally, non-negative curvature in the modified sense is strictly stronger than non-negative curvature in the classical Bakry Emery sense \cite{munch2017remarks,munch2014li}.
Moreover, computing Bakry curvature is computationally simple since it reduces to a semidefinite programming problem \cite{cushing2016bakry}.

Naturally, the question arises, if the Li-Yau inequality and its consequences are also true under non-negative Bakry Emery curvature.
Due to the lack of a chain rule for the Laplacian in discrete settings, this turned out to be a hard question and there was no significant progress in this direction for more than five years.
There even seemed to be a consensus that strong curvature assumptions are necessary for proving strong analytical and geometrical results as Li-Yau inequality, Harnack inequality, volume doubling or Gaussian heat kernel estimates, as there is a growing number of articles investigating the strong curvature notions
\cite{bauer2015li,horn2014volume,
horn2019spacial,dier2017discrete,
munch2017remarks,munch2014li,
gong2017equivalent,
gong2018li,fathi2015curvature,
lin2015equivalent,lin2017existence,
hua2016curvature,gao2016curvature,gao2016one,
Lv2019,lippner2016li,Wang2017,qian2017remarks,
wu2018nonexistence,lin2019ultracontractivity,
wang2018eigenvalue,lin2017log}.

In this paper, we establish a Li-Yau inequality by modifying the heat equation instead of the curvature condition.
The motivation is that the discrete Bakry Emery curvature seems not to be compatible with the logarithm, so we aim to replace $\log P_t f$ by $u_t$ which satisfies a suitable differential equation.
Assuming a chain rule for the Laplacian, one has
\begin{align}\label{eq:dtu}
\partial u_t = \partial_t \log P_t f=\frac{\Delta P_t f}{P_t f} = \frac{\Delta e^{u_t}}{e^{u_t}} = \Delta u_t + \Gamma u_t
\end{align}
where $\Gamma u_t = |\nabla u_t|^2$.
The main idea in this paper is to take \eqref{eq:dtu} as a definition in the discrete setting and use it to approximate $\log P_t f$.

By Picard-Lindel\"of, we have uniqueness and short-time existence of the initial value problem.
Long-time existence cannot be guaranteed in general since the non-linear $\Gamma$ term can make the differential equation explode.
However, we will prove long time existence via a gradient estimate under non-negative Bakry Emery curvature (see Theorem~\ref{thm:GradEst}).
Having established Li-Yau inequality (see Theorem~\ref{thm:LiYau}), we deduce a parabolic Harnack inequality (see Theorem~\ref{thm:Harnack}) which will be used to prove volume doubling (see Theorem~\ref{thm:volDoubling}). By this, we will show that there are no expander graphs satisfying the Bakry-Emery curvature condition $CD(0,n)$ with finite dimension $n$ (see Corollary~\ref{cor:Expanders}).

Let us briefly recall some known consequences of lower bounds on the Bakry Emery curvature.
Eigenvalue estimates in terms of the curvature and diameter are given in \cite{chung2017curvature,lin2010ricci}.
Buser's inequality has been proven in \cite{liu2014eigenvalue,liu2018buser,
liu2015curvature,klartag2015discrete}.
Gaussian concentration has been proven in \cite{schmuckenschlager1998curvature,
fathi2015curvature}.
Optimal diamter bounds under positive curvature and their rigidity have been investigated in \cite{liu2017rigidity,liu2016bakry,
fathi2015curvature}.
Graphs with non-constant lower bounds on the Bakry Emery curvature have been investigated in \cite{munch2018perpetual,liu2017distance}.

\subsection{Setup and Notation}
In this paper, we restrict ourselves to finite graphs. 
A finite graph $G=(V,q)$ consists of a finite set $V$ and a function $q:V^2 \to [0,\infty)$ with $q(x,x)= 0$ for all $x \in V$.
The function $q$ corresponds to the jump rates of the corresponding continuous time Markov chain.
A graph is called reversible or undirected if there exists a function $m: V\to (0,\infty)$ s.t.
\[
q(x,y)m(x)=q(y,x)m(y).
\]
The function $m$ is called reversible measure where we use the terminology of Markov chains.
We  write $x \sim y$ if $q(x,y)>0$ or $q(y,x)>0$.
We remark that this definition is slightly non-standard in the non-reversible case as it might happen that $q(x,y)>0=q(y,x)$.
We define the minimum jump rate
\[
q_{\min}:=\min \{ q(x,y): x,y \in V, \quad q(x,y)>0 \}.
\]
Indeed, most of the results of this article do not require reversibility.
We denote the maximum vertex degree by
\[
D:=\max_x \sum_y q(x,y).
\]
We write $C(V) = \R^V$.
The negative semidefinite Laplacian $\Delta: C(V) \to C(V)$ is given by
\[
\Delta f(x):=\sum_{y} q(x,y)(f(y)-f(x)).
\]
The Laplacian generates the heat semigroup $P_t=e^{t\Delta}$ giving the unique solution to the heat equation $\partial_t P_t f = \Delta P_t f$ and $P_0 f = f$.
The combinatorial graph distance for $x,y \in V$ is given by
\[
d(x,y):= \inf\{n: x=x_0 \sim \ldots \sim x_n=y\}.
\]
We write $B_r(x):=\{y \in V: d(x,y) \leq r\}$.

For reversible graphs with reversible measure $m$, we can define the $\ell_p$ norms.
For $f \in C(V)$, the $\ell_p$ norm w.r.t. $m$ is given by $\|f\|_p := \left(\sum_x |f(x)|^p m(x)\right)^{1/p}$ for $p \in [1,\infty)$ and $\|f\|_\infty := \sup_{x \in V} |f(x)|$. The $\ell_\infty$-norm is also well defined in the non-reversible case.

We recall the Bakry Emery calculus introduced in \cite{BakryEmery85} and first adopted to discrete settings in \cite{schmuckenschlager1998curvature,
lin2010ricci}.
The operator $\Gamma_k : C(V) \times C(V) \to C(V)$ is inductively given by $\Gamma_0(f,g) := fg$ and
\[
2\Gamma_{k+1}(f,g) := \Delta \Gamma_k(f,g) - \Gamma_k(f,\Delta g) - \Gamma_k(\Delta f,g).
\]
We write $\Gamma_k f := \Gamma_k(f,f)$ and $\Gamma := \Gamma_1$.
A graph is said to satisfy the curvature dimension inequality $CD(K,n)$ if for all $f \in C(V)$,
\[
\Gamma_2 f \geq \frac 1 n (\Delta f)^2 + K \Gamma f.
\]
The parameter $K$ can be seen as a lower Ricci curvature bound and $n$ as a local upper dimension bound.

In the paper, we will assume throughout that $(u_t)_{t \in [0,T)} \in C(V)$ is a solution to
\[
\partial_t u_t = \Delta u_t + \Gamma u_t.
\]
Since this equation has a unique solution by Picard Lindel\"of, we use the semigroup notation and write
\[
L_t f := u_t
\]
when assuming $u_0 = f$.

\subsection{Basics on ordinary differential equations}

As we will employ a modified heat equation, we recall some basic facts about autonomous ordinary differential equations for convenience of the reader. This review is mainly based on \cite{grigor2007ODE}.
Let $(U_t)_{t \in \R} \in \R^n$. Let $A:\R^n \to \R^n$ be locally Lipschitz, i.e., there exists $c(R)$ s.t. for all $x,y \in \R^n$, one has $\|A(x)-A(y)\|_2 \leq c(R) \|x-y\|_2$ where $R = \max(\|x\|_2,\|y\|_2)$. Let $u_0 \in \R^n$.
Then by the Picard-Lindel\"of theorem, there exists an interval $(a,b)$ with $0 \in (a,b)$ and a solution to $U_0=u_0$ and
$\partial_t U_t = A(U_t)$ for $t \in (a,b)$ such that every other solution $\widetilde U_t$ on an interval $I$ which contains zero, satisfies $\widetilde U_t = U_t$ on $I \subset (a,b)$.
Moreover, the unique solution $U_t$ continuously depends  on $u_0$ for fixed $t$. Furthermore if $b < \infty$, then $U_t$ leaves every compact set, i.e., the set $\{U_t:t \in [0,b)\}$ is unbounded.

\section{A modified heat equation}

In this section, we give fundamental properties of the modified heat equation. Particularly, we prove gradient estimates, long times existence, monotonicity, and we establish a precise comparison to the linear heat semigroup $P_t$.

\subsection{Gradient estimate}

The following theorem is vital to ensure long time existence of $u_t$. 
It states that the gradient of $u_t$ decays with rate $K$ assuming $CD(K,\infty)$.
We will assume $K\geq 0$ since otherwise we cannot ensure long time existence of $u_t$.

\begin{theorem}[Gradient estimate]\label{thm:GradEst}
Let $G=(V,q)$ be a finite graph satisfying 
$CD(K,\infty)$ for some $K\geq 0$.
Suppose $\|\Gamma u_0\|_\infty \leq q_{\min}/2$. 
Then for all $t\geq 0$, the solution $u_t$ exists and satisfies
\[
\|\Gamma u_t\|_\infty \leq e^{-2Kt} \|\Gamma u_0\|_\infty.
\]
In particular,
$|u_t(y)-u_t(x)| \leq 1$ for all $x\sim y$ and all $t\geq 0$.
\end{theorem}
In case of negative curvature, the same statement can be proven under a stronger assumption on the gradient of $u_0$, however, existence of $u_t$ can only be guaranteed until a fixed time depending on the curvature and the initial gradient bound.
\begin{proof}
We prove the theorem by contradiction.
By the Picard Lindel\"of theorem, one has short time existence and smoothness of the solution since the right side $\Delta + \Gamma$ is smooth.
By continuity of the solution in the initial state $u_0$, we can assume $\|\Gamma u_0\|_\infty < q_{\min}/2$ without loss of generality.
Suppose $\|\Gamma u_t\|_\infty > e^{-2Kt} \|\Gamma u_0\|_\infty$ for some $t >0$. 
Then $\Gamma u_t(x) - \eps t \geq e^{-2Kt} \|\Gamma u_0\|_\infty$ for small $\eps>0$.
Let 
\[
t_0 := \inf\{t>0: \|\Gamma u_t\|_\infty  \geq e^{-2Kt} \|\Gamma u_0\|_\infty + \eps t\}.
\]
Then $t_0 >0$ as $\|\Gamma u_t\|_\infty$ is continuous in $t$ for small $t\geq 0$.
If $\eps$ is small enough, then $\|\Gamma u_{t_0}\|_\infty \leq q_{\min}/2$ since $\|\Gamma u_0\|_\infty < q_{\min}/2$ by assumption.
We write $f:= u_{t_0}$ and observe that 
\[
\inf_{x \colon \Gamma f(x)=\|\Gamma u_{t_0}\|_\infty} \partial_t \Gamma u_{t_0}(x) = \partial_t^-  \|\Gamma u_{t_0}\|_\infty \geq -2K e^{-2Kt} \|\Gamma u_{0}\|_\infty +\eps \geq -2K  \|\Gamma f\|_\infty + \eps
\]
where $\partial_t^-$ is the left derivative and the first inequality follows from $t_0>0$.
On the other hand,
\begin{align*}
\partial_t \Gamma u_{t_0} = 2\Gamma(u_{t_0},\partial_t u_{t_0}) = 2\Gamma(f, \Delta f + \Gamma f) = -2\Gamma_2 f + \Delta \Gamma f + 2 \Gamma(f,\Gamma f) \leq -2K\Gamma f  + \Delta \Gamma f + 2 \Gamma(f,\Gamma f)
\end{align*}
where we used $CD(K,\infty)$ in the last estimate.
If $\Gamma f$ attains its maximum in $x$, then
$-2K \Gamma f(x) = -2K\|\Gamma f\|_\infty$
and
\begin{align*}
\Delta \Gamma f(x) + 2\Gamma(f,\Gamma f)(x) &= \sum_{y\sim x} q(x,y)(\Gamma f(y)-\Gamma f(x)) +\sum_{y\sim x} q(x,y)(\Gamma f(y)-\Gamma f(x))(f(y)-f(x))\\
&=\sum_{y\sim x} q(x,y)(\Gamma f(y)-\Gamma f(x))(1 + f(y)-f(x)).
\end{align*}
Since $\Gamma f$ is maximal in $x$, we have $\Gamma f(y)-\Gamma f(x) \leq 0$ and since $\Gamma f(x) \leq q_{\min}/2$, we have $(f(y)-f(x))^2 \leq 1$ whenever $q(x,y)>0$, implying $1+f(y)-f(x) \geq 0$.
In particular,
\[
\Delta f(x)+ 2 \Gamma(f,\Gamma f)(x) \leq 0.
\]
Putting everything together yields
\[
-2K\|\Gamma f\|_\infty + \eps \leq \partial_t^- \|\Gamma u_{t_0}\|_\infty = \inf_{x \colon \Gamma f(x)=\|\Gamma f\|_\infty} \partial_t \Gamma u_{t_0}(x)  \leq -2K\|\Gamma f\|_\infty.
\]
This is a contradiction proving the desired inequality for all $t$ for which $u_t$ exists.
Finally, long time existence follows since $|\partial_t u_t| \leq D \|u_t\|_\infty + \|\Gamma u_0\|_\infty \leq D \|u_t\|_\infty + q_{\min}/2$ showing that $u_t$ stays within a compact set for all finite $t$. This finishes the proof.
\end{proof}

\begin{remark}
One might be tempted to think that the above gradient estimate characterizes $CD(K,\infty)$.
This however is not true as in the proof, we use maximality of the gradient. Particularly, the graph $G=(\{1,2,3\},q)$ with $q(1,2)=2$ and $q(2,1)=1$ and $q(2,3)=5$ and $q(3,2)=1$ and $q(1,3)=q(3,1)=0$ satisfies $CD(0,\infty)$ with $0$ as optimal curvature bound, but it satisfies the gradient estimate of the theorem with $K=1$ as can be shown in an elementary computation.
\end{remark}

\begin{remark}
For comparison purposes, we recall similar gradient estimates.
\begin{enumerate}[(i)]
\item Lower bounds of the Bakry Emery curvature can be characterized via gradient estimates for the heat semigroup, see \cite{lin2015equivalent,keller2018gradient,
hua2017stochastic,gong2017equivalent}. Particularly, $CD(K,\infty)$ is equivalent to
\[
\Gamma P_t f \leq e^{-2Kt}P_t \Gamma f.
\]
\item
Lower bounds on the Ollivier curvature (see \cite{ollivier2007ricci,ollivier2009ricci}) can also be characterized via a Lipschitz decay of the heat semigroup. Particularly by \cite{munch2017ollivier}, Ollivier curvature bounded from below by $K$ is equivalent to
\[
\|\nabla P_t f \|_\infty \leq e^{-Kt}\|\nabla f \|_\infty.
\]
where $\|\nabla f\|_\infty$ denotes the optimal Lipschitz constant w.r.t the canonical graph distance $d$.
\item
Lower bounds on a modified Ollivier curvature can  be characterized via a Lipschitz decay of the logarithm of the heat semigroup. Particularly by \cite[Theorem~1.8]{kempton2019large}, the exponential Ollivier curvature bounded from below by $K$ is equivalent to
\[
\|\nabla \log P_t f \|_\infty \leq e^{-Kt}\|\nabla \log f \|_\infty.
\]
\item By \cite{erbar2018poincare}, the entropic Ricci curvature bounded from below by $K$ for a reversible graph with reversible measure $m$ is equivalent to
\[
\|\nabla P_t f\|_\rho^2 \leq e^{-2Kt}\|\nabla f\|^2_{P_t \rho}
\]
for all positive $f,\rho \in C(V)$ where
\[
\|\nabla f\|^2_{\rho} := \sum_{y\sim x} q(x,y)m(x)(f(y)-f(x))^2 \frac{\rho(y)-\rho(x)}{\log \rho(y)- \log \rho(x)}
\]
and the latter term is set to be $\rho(x)$ if $\rho(y)=\rho(x)$.
\end{enumerate}

\end{remark}

\subsection{Monotonicity}
We recall that $L_t f$ is the unique maximal solution to
$L_0 f = f$ and
\[
\partial_t L_t f = \Delta L_t f + \Gamma L_t f.
\]
As we know that the heat semigroup $P_t$ is monotonous, we expect the same for $L_t$. This indeed holds true when assuming small enough gradients of the initial values as shown in the following theorem.

\begin{theorem}[Monotonicity]\label{thm:Monotonicity}
Let $G=(V,q)$ be a finite graph satisfying 
$CD(K,\infty)$ for some $K\geq 0$.
Suppose $\|\Gamma f\|_\infty \leq q_{\min}/2$.
Further suppose $g \geq f$. 
Then for all $t\geq 0$ for which $L_t g$ exists,
\[
L_t g \geq L_t f.
\]
\end{theorem}
\begin{remark}
The condition $\|\Gamma f\|_\infty \leq q_{\min}/2$ is necessary since $L_t f$ can explode otherwise and $L_t g$ can stay constant.
\end{remark}
\begin{proof}[Proof of Theorem~\ref{thm:Monotonicity}]
Suppose $L_t g(x) <L_t f(x)$ for some $x$ and $t$. Let $\eps>0$ be small enough s.t.
\[
t_0 := \inf\{t:L_t g(x) +\eps t < L_t f(x) \mbox{ for some }x\}
\]
is finite.
Let $x_0$ s.t. $L_t g(x_0) +\eps t = L_t f(x_0)$.
We observe 
$$
f_y:=L_{t_0} f(y) - L_{t_0}f(x_0) \leq L_{t_0} g(y) - L_{t_0}g(x_0) =:g_y.
$$
Thus,
\[
\sum_{y}q(x_0,y) f_y\left(1+\frac{f_y}2  \right)=
\partial_t L_t f(x_0)  \geq \eps +\partial_t L_t g(x_0) = \eps \sum_{y}q(x_0,y) g_y\left(1+\frac{g_y}2  \right).
\]
By Theorem~\ref{thm:GradEst}, we have $g_y \geq f_y \geq -1$ for all $y \sim x_0$ and thus,  
\[
f_y\left(1+\frac{f_y}2  \right) \leq g_y\left(1+\frac{g_y}2  \right)
\]
for all $y \sim x_0$. This is a contradiction to the above inequality which finishes the proof.
\end{proof}

The monotonicity gives another justification that the non-linear semigroup $L_t$ is a suitable substitute for $\log P_t$. However, the striking argument for the analogy between $L_t$ and $\log P_t$ is given in the next subsection.

\subsection{Semigroup and $\ell_1$-norm comparison}

We now compare $u_t$ with $\log P_t f$ in a precise way.
Particularly, $e^{u_t}$ can be be upper and lower bounded by
$(P_t e^{\alpha u_0})^{1/\alpha}$
with suitable choices of $\alpha$ as shown in the following theorem. 

\begin{theorem}[Semigroup comparison]\label{thm:SemigroupCompare}
Let $G=(V,q)$ be a finite graph satisfying $CD(0,\infty)$. Suppose $\Gamma u_0 \leq q_{\min}/2$.
Then,
\[
P_t e^{\alpha u_0} \geq e^{\alpha u_t} \qquad \mbox{ for } \alpha \geq 1.60,
\]
and
\[
P_t e^{\alpha u_0} \leq e^{\alpha u_t} \qquad \mbox{ for } \alpha \leq 0.76.
\]
\end{theorem}
\begin{proof}
Let
\[
G(s):=P_{t-s}e^{\alpha u_s}.
\]
Then,
\begin{align*}
\partial_s G(s) = P_{t-s} \left(\partial_s e^{\alpha u_s} - \Delta e^{\alpha u_s} \right)
= P_{t-s} \left(\alpha\partial_s(u_s) e^{\alpha u_s} - \Delta e^{\alpha u_s} \right) = P_{t-s} \left(\alpha(\Delta u_s +\Gamma u_s) e^{\alpha u_s} - \Delta e^{\alpha u_s} \right)
\end{align*}
We fix $x \in V$. With $t(y):=u_s(y)-u_s(x)$, we calculate
\[
\left(\alpha(\Delta u_s +\Gamma u_s) e^{\alpha u_s} - \Delta e^{\alpha u_s} \right) (x) = e^{\alpha u_s}(x)\sum_y q(x,y) 
\left[\alpha t(y)\left(1+ \frac {t(y)}2\right) - \left(e^{\alpha t(y)}-1 \right) \right].
\]
By Theorem~\ref{thm:GradEst}, we have $|t(y)|\leq 1$ for all $y \sim x$. For $\alpha \geq 1.60$ and $|t|\leq 1$, one has
\[
\alpha t(1+t/2) - (e^{\alpha t}-1) \leq 0
\]
and for 
$\alpha \leq 0.76$ and $|t|\leq 1$, one has
\[
\alpha t(1+t/2) - (e^{\alpha t}-1) \geq 0.
\]
Thus, for $\alpha \geq 1.60$, we have $\partial_s G(s) \leq 0$ yielding $P_t e^{\alpha u_0} = G(0) \geq G(t) = e^{\alpha u_t}$. Analogously, we obtain $P_t e^{\alpha u_0} \leq e^{\alpha u_t}$ in the case $\alpha \leq 0.76$.
This finishes the proof.
\end{proof}

As the heat semigroup of a reversible graph preserves the $\ell_1$ norm, we can already conclude that $\|e^{\alpha u_t}\|_1$ is decreasing in $t$ for $\alpha \geq 1.6$ and increasing for $\alpha \leq 0.76$.
The following theorem improves the bounds for $\alpha$ to $1.1$ and $1$ respectively.

\begin{theorem}[$\ell_1$-comparison]
\label{thm:ell1Compare}
Let $G=(V,q)$ be a finite reversible graph satisfying $CD(0,\infty)$. Let $m$ be the reversible measure. Suppose $\Gamma u_0 \leq q_{\min}/2$.
Then,
\[
\|e^{\alpha u_0}\|_1 \geq \|e^{\alpha u_t}\|_1 \qquad \mbox{ for } \alpha \geq \log 3 ,
\]
and
\[
\| e^{\alpha u_0}\|_1 \leq \|e^{\alpha u_t}\|_1 \qquad \mbox{ for } \alpha \leq 1.
\]
\end{theorem}

\begin{proof}
Let 
\[
G(s):=\|e^{\alpha u_s}\|_1.
\]
We write $w(x,y)=q(x,y)m(x)=q(y,x)m(y)$
Then,
\begin{align*}
\partial_s G(s) &= \langle \alpha \partial_s u_s,  e^{\alpha u_s} \rangle = \alpha \langle \Delta u_s + \Gamma u_s, e^{\alpha u_s} \rangle = \frac\alpha 2 \sum_{y\sim x}w(x,y) e^{\alpha u_s(x)}(u_s(y) - u_s(x))\left(2+u_s(y)-u_s(x) \right) \\
&=
\frac \alpha 4 \sum_{y\sim x} 
w(x,y)
\left[ e^{\alpha u_s(x)}(u_s(y) - u_s(x))\left(2+u_s(y)-u_s(x) \right)\right]  \\& \qquad\quad +w(x,y)
\left[ e^{\alpha u_s(y)}(u_s(x) - u_s(y))\left(2+u_s(x)-u_s(y) \right)\right] \\
&= \frac \alpha 4 \sum_{y\sim x}w(x,y)
\exp\left(\frac{\alpha(u_s(x)+u_s(y))}{2}\right) \cdot \left[ 
e^{-\alpha t/2}\cdot t (2+t) + e^{\alpha t/2} \cdot (-t) \cdot (2-t)
\right]
\end{align*}
where $t=t(s,x,y)=u_s(y)-u_s(x) \in [-1,1
]$ for all $y\sim x$ by Theorem~\ref{thm:GradEst}. An elementary calculation shows that for $t \in [-1,1]$, the expression
\[
e^{-\alpha t/2}\cdot t (2+t) + e^{\alpha t/2} \cdot (-t) \cdot (2-t)
\]
is non-negative if $\alpha \leq 1$ and non-positive if $\alpha \geq \log 3$.
Thus, $\partial_s G(s)$ is non-negative if $\alpha \leq 1$ and non-positive if $\alpha \geq \log 3$ which directly implies the claim of the theorem.
\end{proof}

We remark that in both Theorem~\ref{thm:ell1Compare} and Theorem~\ref{thm:ell1Compare}, the bounds for $\alpha$ can be even improved when assuming that $\Gamma u_0$ is bounded by a smaller constant.
We will use the semigroup comparison together with the Harnack inequality to prove volume doubling. The Harnack inequality will be proven in the next section.

\section{Li-Yau-Hamilton-Harnack type inequalities}

The gradient estimate $|u_t(y)-u_t(x)| \leq 1$ for $y \sim x$ from Theorem~\ref{thm:GradEst} turns out to be essential for proving various Li-Yau-Hamilton-Harnack type inequalities.

\subsection{Li-Yau inequality}

Under several curvature assumptions which seem hard to verify, a Li-Yau inequality has been proven for the heat semigroup \cite{bauer2015li,horn2014volume,munch2014li,
dier2017discrete,gong2018li}.
Using the gradient estimate from Theorem~\ref{thm:GradEst}, we prove a Li-Yau inequality under $CD(0,n)$ for the modified heat equation
\[
\partial_t u_t = \Delta u_t + \Gamma u_t.
\]
We recall that we consider $u_t$ as a substitute for $\log P_t f$.

\begin{theorem}[Li-Yau inequality]\label{thm:LiYau}
Let $G=(V,q)$ be a finite graph satisfying 
$CD(0,n)$ for some $n<\infty $.
Suppose $\|\Gamma u_0\|_\infty \leq q_{\min}/2$. 
Then, 
\[
\Gamma u_t - \partial_t u_t = -\Delta u_t \leq \frac{n}{2t}.
\]
\end{theorem}
\begin{proof}
We fix $T>0$.
Let $F=-t\Delta u$ on $V \times [0,T]$. 
Suppose $\max F(x,t) > n/2$. Let the maximum be attained at $(x,t)$ which is fixed from now on.
Then, $t>0$ and by $CD(0,n)$,
\begin{align*}
0\leq \partial_t F(x,t) = -t \Delta \partial_t u_t(x) - \Delta u_t(x) &= - t\Big(\Delta \Delta u_t(x)+ \Delta \Gamma u_t(x)\Big) - \Delta u_t(x)\\
&\leq -t\left(\Delta \Delta u_t(x) + 2\Gamma(u_t,\Delta u_t)(x) + \frac 2 n(\Delta u_t(x))^2\right) - \Delta u_t(x).
\end{align*}
We observe
\begin{align*}
\Delta\Delta u_t(x) + 2\Gamma(u_t,\Delta u_t)(x) &= \sum_y q(x,y) (\Delta u_t(y)-\Delta u_t (x))  + \sum_{y}q(x,y)(\Delta u_t(y)-\Delta u_t (x))(u_t(y)-u_t(x))\\
&=\sum_{y}q(x,y)(\Delta u_t(y)-\Delta u_t (x))(1+u_t(y)-u_t(x)) \geq 0
\end{align*}
since $\Delta u_t(y)-\Delta u_t (x) \geq 0$ for  by minimality of $\Delta u_t(x)$ and since $1+u_t(y)-u_t(x) \geq 0$ for $y \sim x$ by Theorem~\ref{thm:GradEst}.
Thus,
\begin{align*}
0\leq \partial_t F(x,t) \leq -\frac{2t}{n}(\Delta u_t(x))^2 - \Delta u_t(x).
\end{align*}
Rearranging gives $F(x,t)= -t\Delta u_t(x) \leq 2n $ and thus, $-\Delta u_t \leq \frac n{2t}$.  
This finishes the proof.
\end{proof}

\subsection{Harnack inequality}

One of the main application of Li-Yau inequality is the parabolic Harnack inequality which can be considered as in integrated version of Li-Yau inequality. The proof of the following theorem follows \cite{munch2014li} and \cite{bauer2015li}.

\begin{theorem}[Harnack inequality]\label{thm:Harnack}
Let $G=(V,q)$ be a finite graph satisfying 
$CD(0,n)$ for some $n<\infty $.
Suppose $\|\Gamma u_0\|_\infty \leq q_{\min}/2$. 
Then, for $x,y\in V$ and $T_2>T_1 >0$,
\[
u_{T_1}(x)-u_{T_2}(y) \leq \frac n 2 \log \left(\frac{T_2}{T_1}\right) + \frac{2d(x,y)^2}{q_{\min}(T_2-T_1)}
\]
\end{theorem}

\begin{proof}
We first prove the claim for $x\sim y$. We first assume $q(y,x)>0$. By Li-Yau inequality (Theorem~\ref{thm:LiYau}),
\begin{align*}
u_{T_1}(x)-u_{T_2}(y) &= -\int_{T_1}^s \partial_t u_t(x) dt + (u_s(y)-u_s(x)) - \int_{s}^{T_2}\partial_t u_t(y)dt
\\&\leq \int_{T_1}^{T_2} \frac n{2t} dt - \int_{s}^{T_2} \Gamma u_t(y) dt + (u_s(y)-u_s(x))
\\& 
\leq 
\frac n 2 \log \left(\frac{T_2}{T_1}\right) - \int_{s}^{T_2} \Gamma u_t(y) dt  + 
\sqrt{\frac{2\Gamma u_s(y)}{q_{\min}}}.
\end{align*}
Minimizing over $s \in [T_1,T_2]$ and applying \cite[Lemma~5.3]{munch2014li} gives
\[
u_{T_1}(x)-u_{T_2}(y) \leq \frac n 2 \log \left(\frac{T_2}{T_1}\right) + \frac{2}{q_{\min}(T_2-T_1)}
\]
as desired.
In the case $q(x,y)>0=q(y,x)$, we keep the $\int_{T_1}^s\Gamma u_t(x)$ term instead of the $\int_{s}^{T_2} \Gamma u_t(y)$ term, and we estimate $u_s(y)-u_s(x)$ by the gradient at $x$ instead of $y$. By doing so, we end up with the same inequality for $u_{T_1}(x)-u_{T_2}(y)$.

The general case, i.e. when $x$ and $y$ are not adjacent, follows by subdividing the time interval $[T_1,T_2]$ in $d(x,y)$ many subintervals of the same size and by applying the above inequality successively. This finishes the proof.
\end{proof}


\subsection{Hamilton type estimate}
The Hamilton estimate (see \cite{hamilton1993matrix}) states that on Riemannian manifolds with non-negative Ricci curvature, one has for all positive functions $f$,
\[
\Gamma \log P_t f \leq \frac 1 t \log \left(\frac{\|f\|_\infty}{P_t f} \right).
\]
A discrete version of this estimate under a modified Bakry Emery condition has been established in \cite{horn2019spacial}.
Here, we prove the Hamilton type estimate for the modified heat equation under the $CD(K,\infty)$ condition.
\begin{theorem}[Hamilton gradient estimate]\label{thm:Hamilton}
Let $G=(V,q)$ be a finite graph satisfying $CD(K,\infty)$ for some $K\geq  0$. Suppose $u_0 \leq 0$ and $\Gamma u_0 \leq \frac {q_{\min}}2$. Then,
\[
\Gamma u_t \leq -\frac{u_t}{\phi(t)}.
\]
with
\[
\phi(t)= \begin{cases}
\frac{e^{2Kt}-1}{2K} &: K>0 \\
t&: K=0.
\end{cases}
\]
\end{theorem}

\begin{proof}
W.l.o.g., we can assume that $\sup u_0 < 0$ giving $\sup_{t} u_t <0$ by Theorem~\ref{thm:Monotonicity}. Let 
\[
H(\cdot,t):=-\frac{\phi(t)\Gamma u_t}{u_t}
\]
be defined on $V\times [0,T]$ for some fixed $T>0$. Assume the maximum of $H$ is attained at $(x,t)$. Denote the maximum by $C$.
W.l.o.g., we have $C>0$ and $t>0$ giving $\partial_t H(x,t) \geq 0$.
Moreover, $\Gamma u(x)=-Cu(x)/\phi(t)$ and $\Gamma u \leq - Cu/\phi(t)$. We aim to show $C\leq 1$ which will be done by computing $\partial_t H$ in the maximum point and by systematically replacing or estimating $\Gamma u$ by $-Cu/\phi(t)$.
By Theorem~\ref{thm:GradEst}, we have
 $1+u(y)-u(x)\geq 0$ for $y\sim x$ and thus at the maximum point $(x,t)$,
\begin{align*}
2K\Gamma u(x) + \partial_t \Gamma u(x) &= 2K\Gamma u(x)+
2\Gamma(u,\Delta u + \Gamma u) (x) \\
&\leq \Delta \Gamma u(x) + 2\Gamma(u,\Gamma u)(x)
\\&=
 \sum_{y\sim x}q(x,y)(\Gamma u(y)-\Gamma u(x))(1+u(y)-u(x)) \\
&\leq -\frac C {\phi(t)} \sum_{y\sim x}q(x,y)(u(y)-u(x))(1+u(y)-u(x)) \\
&=-\frac C {\phi(t)} \left( \Delta u + 2 \Gamma u\right)(x)\\
&=-\frac C {\phi(t)} \partial_t u(x) + \frac{C^2}{{\phi(t)}^2}u(x).
\end{align*}
where we applied $CD(0,\infty)$ in the first estimate and the inequality $1+u(y)-u(x) \geq 0$ in the second estimate.
Hence,
\[
\partial_t \Gamma u(x) \leq -\frac C {\phi(t)} \partial_t u(x) + \frac{C^2}{{\phi(t)}^2}u(x) + \frac{2KC}{\phi(t)}u(x).
\]
We now estimate the time derivative of $H$ at $(x,t)$, 
\begin{align*}
0 \leq \partial_t H &= -\phi'(t)\frac{\Gamma u}{u} - \phi(t)\left(\frac{\partial_t \Gamma u}u  - \frac {\Gamma u \partial_t u}{u^2}\right)\\
&=\frac {C\phi'(t)}{\phi(t)} - \phi(t) \left(\frac{\partial_t \Gamma u}u +\frac C {\phi(t)} \frac {\partial_t u}{u} \right) \\
&\leq \frac {C\phi'(t)}{\phi(t)} - \phi(t) \left(\frac{-\frac C {\phi(t)} \partial_t u + \frac{C^2}{\phi(t)^2}u + \frac{2KC}{\phi(t)}u}u +\frac C {\phi(t)} \frac {\partial_t u}{u} \right) \\
&=\frac{C(\phi'(t)-C-2K\phi(t))}{\phi(t)}\\
&=\frac{C(1-C)}{\phi(t)}
\end{align*}
which clearly gives $C \leq 1$. Hence, $H\leq 1$ which implies
\[
\Gamma u_t \leq -\frac {u_t} {\phi(t)}
\]
as desired. This finishes the proof.
\end{proof}
By integrating the Hamilton estimate over space, we also get a corresponding Harnack inequality.

\begin{theorem}[Hamilton Harnack inequality]\label{thm:HamiltonHarnack}
Let $G=(V,q)$ be a finite graph satisfying $CD(0,\infty)$. Suppose $q(x,y)>0$ for all $x\sim y$.
 Suppose $u_0 \leq 0$ and $\Gamma u_0 \leq \frac {q_{\min}}2$. Then for all $x,y \in V$ and all $t>0$,
\[
\left|\sqrt{-u_t(y)}- \sqrt{-u_t(x)} \right| \leq \frac{d(x,y)}{\sqrt{2tq_{\min}}}.
\]
\end{theorem}
\begin{proof}
For simplicity, we write $u$ instead of $u_t$.

 We can assume $x\sim y$ and $-u(y) \geq -u(x)$ without loss of generality since we assumed $q(x,y)>0$ for all $x\sim y$.
Thus, we have
\begin{align*}
\left(\sqrt{-u(y)}-\sqrt{-u(x)}\right)\left(\sqrt{-u(y)}+\sqrt{-u(x)}\right) =u(x)-u(y) \leq \sqrt{\frac{ 2\Gamma u(x)}{q_{\min}}} \leq  \sqrt{\frac{ -2u(x)}{tq_{\min}}} \leq \frac{\sqrt{-u(y)}+\sqrt{-u(x)}}{\sqrt{2tq_{\min}}}
\end{align*}
where we applied Theorem~\ref{thm:Hamilton} in the second estimate.  Dividing by $\left(\sqrt{-u(y)}+\sqrt{-u(x)}\right)$ gives
\begin{align*}
\left| \sqrt{-u(y)}-\sqrt{-u(x)} \right| \leq \frac{1}{\sqrt{2tq_{\min}}}
\end{align*}
which finishes the proof.
\end{proof}

\begin{remark}
A similar estimate was shown in \cite[Theorem~2.5.2]{munch2019non} under non-negative Ollivier curvature, namely
\[
|P_t f(y)-P_t f(x)| \leq \frac{d(x,y)}{\sqrt{tq_{\min}}} \cdot \|f\|_\infty.
\]
In case of non-negative Bakry Emery curvature, there is a corresponding gradient estimate stating
\[
\Gamma P_t f \leq \frac{1}{2t}\left(P_t f^2 - (P_t f)^2 \right),
\]
see e.g. \cite{liu2014eigenvalue,
klartag2015discrete,lin2015equivalent}.
\end{remark}

\section{Applications}

As applications of Li-Yau inequality, we prove volume doubling and show that there exist no expander graphs satisfying $CD(0,n)$.

\subsection{Volume doubling}

One of the major questions regarding discrete Ricci curvature is whether non-negative Bakry-Emery curvature implies volume doubling.
An affirmative answer was given for the special case of normalized birth death chains in \cite{hua2017ricci}. However the general case stayed open. In this section, we prove volume doubling.

Let us briefly recall the proof idea from the Riemannian setting.
We take
\[
f:=1_{B_r(x)}.
\]
Then with $t = C r^2$,
\[
P_t f(x) \geq \frac 1 2
\]
and with $T=2t$ and Harnack inequality,
\[
P_T f(y) \geq P_t f(x) \cdot \left(\frac{T}{t}\right)^{n/2} \exp \left(\frac{d(x,y)^2}{4(T-t)} \right) \geq c>0
\]
for all $y \in B_{2r}(x)$ where $c$ is independent from $r$.
Thus,
\[
\vol(B_r(x))= \|f\|_1=\|P_t f\|_1 \geq \|c 1_{B_{2r(x)}}\|_1 = c \vol(B_{2r}(x)).
\]
Our proof of volume doubling follows the same idea, however we take a detour via the modified heat equation to apply Harnack inequality. 

\begin{theorem}[Volume doubling]\label{thm:volDoubling}
Let $G=(V,q)$ be a finite reversible graph satisfying $CD(0,n)$ for some $n<\infty$. Let $r \geq 4n^2 D/q_{\min}$ and $x \in V$. Then,
\[
\frac{m(B_{2r}(x))}{m(B_{r}(x))} \leq \left( 9n \sqrt{\frac {D}{q_{\min}}}\right)^{3 n}. 
\]
\end{theorem}

\begin{proof}
Let $r\cdot \sqrt{\frac{q_{\min}}{D}} \geq C > 0$, let $x\in V$ and let $u_t \in C(V)$ be given by $\partial_t u_t = \Delta u_t + \Gamma u_t$ and

$$u_0:=-\frac{C}{r}d(x,\cdot) \vee (-C).$$
We remark $\Gamma u_0 \leq \frac {q_{\min}}2$ so we can apply the theory established in this paper.
 We have 
$
e^{\alpha u_0} \geq \alpha u_0 + 1
$ 
and thus, by Theorem~\ref{thm:SemigroupCompare}, we have
\begin{align}\label{eq:LtPtExp}
e^{\alpha u_t} \geq P_t{e^{\alpha u_0}} \geq   \alpha P_t u_0 + 1
\end{align}
for $\alpha = 0.76$.
We know
$
\Gamma u_0 \leq \frac{DC^2}{2r^2}.
$
Thus, by \cite[Theorem~3.1]{lin2015equivalent},
\[
(\Delta P_t u_0)^2 \leq 
\frac{n}{2t} (P_t \Gamma u_0 - \Gamma P_t u_0)
\leq 
\frac{n}{2t}\|\Gamma u_0\|_\infty \leq \frac{nDC^2}{4tr^2}
\]
giving
$
|P_t u_0 - u_0| \leq \frac{C\sqrt{nDt}}r
$
and hence
$
P_t u_0(x) \geq  - \frac{C\sqrt{nDt}}r.
$
Combining with \eqref{eq:LtPtExp}, we obtain

\[
u_t(x) \geq \frac 1 \alpha \log \left(1 - \alpha\frac{C\sqrt{nDt}}r\right). 
\]
For $d(x,y) \leq R:=2r$ and $t<T$, we have by Theorem~\ref{thm:Harnack},
\[
u_T(y) - u_t (x) \geq -\frac n 2 \log\left(\frac{T}{t}\right) - \frac{2R^2}{q_{\min} (T-t)}
\]
We now aim to find $t$ and $T$ giving the optimal estimate for $u_T(y)$.
We set
\[
\sqrt{t}:= \frac{r}{\alpha C \sqrt{nD}} \cdot \frac{{\alpha n}}{\alpha n + 1}
\]
and get
\begin{align*}
\frac 1 \alpha \log \left(1 - \alpha\frac{C\sqrt{nDt}}r\right) + \frac n 2 \log{t} &= -\frac{1}\alpha \log(\alpha n + 1) + n \log\left( \frac{r}{\alpha C\sqrt{nD}} \right) - n\log\left(\frac{\alpha n+1}{\alpha n } \right) \\
&\geq -n-\frac 1 \alpha  +    n \log\left( \frac{r}{\alpha C\sqrt{nD}} \right)
\end{align*}

When setting $T:=\frac{4R^2}{n q_{\min}}$ and choosing $C$ s.t. $T/2 \geq t$, we get
\[
-\frac n 2 \log T - \frac{2R^2}{q_{\min} (T-t)} \geq -n \log \left( \frac{2R}{\sqrt{nq_{\min}}} \right) - n.
\]
Putting together gives
\[
- u_T(y) \leq 2n + \frac 1 \alpha + n \log\left( \frac{2\alpha C R}{r} \sqrt{\frac D{q_{\min}}} \right) =: Q.
\]
By Theorem~\ref{thm:ell1Compare} for $\beta = \log 3$ we have
\begin{align}\label{eq:ebetaLup}
\|e^{\beta u_T}\|_1 \leq \|e^{\beta u_0}\|_1 \leq \|e^{-\beta C} + (1-e^{-\beta C})1_{B_{r}(x)}\|_1 \leq e^{-\beta C}m(V) +  m(B_r(x)).
\end{align}
On the other hand,
\[
e^{\beta u_T} \geq e^{-\beta C} + (e^{-\beta Q}-e^{-\beta C}) 1_{B_R(x)}
\]
and thus,
\begin{align}\label{eq:ebetaLdown}
\|e^{\beta u_T}\|_1 \geq   e^{-\beta C}m(V)  +  (e^{-\beta Q}-e^{-\beta C})  m(B_R(x))
\end{align}

We set $C:= \gamma n$ with
\[
\gamma:=\frac{2\alpha n R}{r} \sqrt{\frac{D}{q_{\min}}}
\]
and get
\begin{align*}
Q-C  &= 2  n +  \frac 1 {\alpha} + 2n\log \left(\gamma \right)   -    \gamma n \\
&\leq n(3 + 2\log \gamma - \gamma)
\end{align*}
We remark that the assumption $r \geq 4n^2 \sqrt{\frac{D}{q_{\min}}}$ ensures $\Gamma u_0 \leq \frac{q_{\min}} 2$ with our choice of $C$.

Since $\alpha=0.76$ and $R/r=2$ and $n \geq 2$ and $D/q_{\min} \geq 2$ we get $\gamma \geq 8.5$ and $n(3 + 2\log \gamma - \gamma) \leq - 2$.
Thus by combining \eqref{eq:ebetaLup} and \eqref{eq:ebetaLdown},
\[
\frac{m(B_r(x))}{m(B_R(x))} \geq
e^{-\beta Q} - e^{-\beta C} = e^{-\beta Q} \cdot \left(1 - e^{\beta(Q-C)} \right) \geq e^{-\beta Q}(1-e^{-2\beta}) \geq 0.86 e^{-\beta Q}.
\]
Taking inverse gives
\begin{align*}
\frac{m(B_R(x))}{m(B_r(x))} &\leq 1.17 e^{2n\beta+\frac {\beta}{\alpha}} \left({\frac{2\alpha n R}{r}  \sqrt{\frac {D}{q_{\min}}}}\right)^{2\beta n}
\\&
\leq 1.17 \cdot 4.26   \left({\frac{4.14 n R}{r}  \sqrt{\frac {D}{q_{\min}}}}\right)^{2\beta n}
\\
&\leq  \left( 9n \sqrt{\frac {D}{q_{\min}}}\right)^{3 n}
\end{align*}
which finishes the proof.
\end{proof}

\begin{remark}
Although the theorem only states the volume doubling for large radii, we get an a priori volume doubling constant for small radii in terms of the parameters $D$, $q_{\min}$ and the maximal radius. Particularly, the global volume doubling constant can be upper bounded only in terms of the dimension $n$, the maximal vertex degree $D$ and the minimum positive jump weight $q_{\min}$.
\end{remark}



\subsection{Fast volume growth for small radii}

One might hope for a volume doubling constant only depending on the dimension and being valid also for small radii. This however does not hold true as the following example demonstrates.

\begin{example}
Let $\eps>0$ and $G_\eps=(V,q)$ with $V=\{1,2,3\}$ and
\begin{align*}
q(1,2)&=\eps \\
q(2,3)&=1 \\
q(2,1)&=4 \\
q(3,2)&=4 \\
q(1,3)=q(3,1)&=0.
\end{align*}
It can be verified e.g. via \cite[Proposition~2.1]{hua2017ricci} that $G_\eps$ satisfies $CD(\frac 1 4,32)$ independently of the choice of $\eps$.
However, the reversible measure $m$ satisfies
\[
\frac{m(1)}{m(2)}=\frac{q(2,1)}{q(1,2)} = \frac{4}{\eps}
\]
and similarly, $\frac{m(2)}{m(3)}=4$
showing that
 \[
\frac{m(B_2(3))}{m(B_1(3))} \longrightarrow \infty \qquad \mbox{ as } \eps \to 0
 \]
although the dimension and curvature in the $CD$ condition stay same. This shows that a uniform volume doubling with a doubling constant only depending on the dimension cannot hold true, not even in case of positive curvature.
\end{example}

\subsection{No Expander graphs with $CD(0,n)$}

An expander graph family is a growing family of combinatorial undirected graphs $(G_i)$ of constant degree $D$ s.t. the first positive eigenvalue of $-\Delta_{G_i}$ is uniformly lower bounded by some $\lambda >0$.
It is conjectured in \cite{cushing2016bakry} that there is no expander graph family satisfying $CD(0,\infty)$. Using the volume doubling property, we can treat the case of finite dimension.

\begin{corollary}\label{cor:Expanders}
Let $n<\infty$. Then
there is no expander graph family satisfying $CD(0,n)$.
\end{corollary}
\begin{proof}
Let Suppose $(G_i)_i$ is a graph expander family satisfying $CD(0,n)$. 
By \cite[Proposition~2.5]{peled2013lipschitz}, we have
\[
|B_r(x)| \geq   \min \left[ \frac{|V_i|}2 ,  C ^{r} \right]
\]
where $C>1$ only depends on the spectral gap and the degree. In particular, balls have exponential volume growth.
This contradicts the polynomial volume growth of balls following from Theorem~\ref{thm:volDoubling}.
This finishes the proof.
\end{proof}

The question about expanders satisfying $CD(0,\infty)$ seems not to be tackleable with purely analytic methods as the birth death chain with constant jump rates to one side and constant but different jump rates to the other side has all analytical properties one expects from expanders, and it satisfies $CD(0,\infty)$.

\printbibliography

\end{document}